\documentclass[a4paper,11pt]{amsart}

\usepackage{amstext}
\usepackage{amsmath}
\usepackage{ifthen}
\usepackage{amscd}
\usepackage{amssymb}
\usepackage[french,english]{babel}
\usepackage[T1]{fontenc}
\usepackage[utf8]{inputenc}
\usepackage[all]{xy}
\usepackage{graphicx}
\usepackage{enumerate}
\usepackage{xspace}
\usepackage{pdfsync}
\usepackage{epic}
\usepackage{dsfont}
\usepackage{stackrel}

\newtheoremstyle{break}
  {}
  {}
  {\itshape}
  {}
  {\bfseries}
  {.}
  {\newline}
  {}
  
\newtheoremstyle{rq}
  {}
  {}
  {\slshape}
  {}
  {\bfseries}
  {.}
  {3pt}
  {}

\newtheoremstyle{exemple}
  {}
  {}
  {\upshape}
  {}
  {\bfseries}
  {.}
  {3pt}
  {}

\newtheoremstyle{fact}
  {}
  {}
  {\slshape}
  {}
  {\bfseries}
  {.}
  {2pt}
  {}

\theoremstyle{fact}
\newtheorem*{fact*}{Fact}
\theoremstyle{break}

\newtheorem{thm}{Theorem}[section]

\newtheorem{cor}[thm]{Corollary}
\newtheorem{lem}[thm]{Lemma}
\newtheorem{prop}[thm]{Proposition}
\newtheorem{defi}[thm]{Definition}
\theoremstyle{rq}
\newtheorem{rem}[thm]{Remark}
\newtheorem{qt}{Question}
\theoremstyle{exemple}

\newcommand{\RR}{\mathbb R}

\newcommand{\CC}{\mathbb C}
\newcommand{\QQ}{\mathbb Q}
\newcommand{\ZZ}{\mathbb Z}

\DeclareMathOperator{\coh}{H}
\DeclareMathOperator{\tr}{tr}
\DeclareMathOperator{\cc}{c}
\newcommand{\HH}{\mathbb{H}}
\newcommand{\cH}{\mathcal{H}}

\newcommand{\cF}{\mathcal{F}}
\newcommand{\cE}{\mathcal{E}}
\newcommand{\cJ}{\mathcal{J}}

\newcommand{\vsh}{\textsc{vhs}}
\newcommand{\drel}{\mathrm{d}_{X\mid B}}

\newcommand{\To}{\longrightarrow}

\newcommand{\set}[1]{\left\{#1\right\}}

\title[Linear K\"ahler groups]{Smooth families of tori and linear K\"ahler groups}
\date{\today}
\author{Beno\^it Claudon}
\address{Universit\'e de Lorraine, Institut \'Elie Cartan Nancy, UMR 7502, B.P. 70239, 54506 Vand\oe uvre-l\`es-Nancy Cedex, France}
\email{Benoit.Claudon@univ-lorraine.fr}
\thanks{It is our pleasure to thank D. Mégy for interesting discussions on the Hodge theoretical part of this paper. We are also grateful to F. Campana and P. Eyssidieux for having shared their insights during our previous collaborations. We finally are indebted to the referee for his careful reading and for pointing out a flaw in the first version of this article.}
\begin{document}

\maketitle

\selectlanguage{french}
\begin{abstract}
Cette courte note améliore les résultats de l'article \cite{CCE2} et peut donc être considérée comme un \emph{addendum} à ce dernier. Nous y établissons qu'un groupe kählérien linéiare peut être réalisé comme le groupe fondamental d'une variété projective lisse. Pour y parvenir, nous étudions certaines déformations relatives de l'espace total d'une famille lisse de tores, et ce dans un contexte équivariant.
\end{abstract}

\selectlanguage{english}

\begin{abstract}
That short note, meant as an \emph{addendum} to \cite{CCE2}, enhances the results contained in \emph{loc. cit.} In particular it is proven here that a linear Kähler group is already the fundamental group of a smooth complex projective variety. This is achieved by studying the relative deformation of the total space of a smooth family of tori in an equivariant context.
\end{abstract}

\section{Introduction}\label{sec:intro}

In his seminal paper on compact complex surfaces \cite{Kod}, Kodaira proved that a compact Kähler surface can be deformed to an algebraic one (Theorem 16.1 in \emph{loc. cit.}). However since the groundbreaking works of Voisin \cite{V04,V06} we know that this is specific to the surface case: in dimension at least 4, there exists compact Kähler manifolds which do not have the cohomology algebra of a projective manifold (and in particular cannot be deformed to such an algebraic manifold). The examples of Voisin being bimeromorphic to a torus (or to a projective bundle over a torus), it leaves open the following question concerning the fundamental groups of compact Kähler manifolds (known as Kähler groups).
\begin{qt}\label{qt:kahler group}
Can any Kähler group be realized as the fundamental group of a smooth complex projective variety? In other terms, is any Kähler group already a projective one?
\end{qt}

Going back to Kodaira's Theorem, Buchdahl gave another proof of this result in \cite{B06,B08}, providing by the way a useful criterion ensuring that a compact Kähler manifold can be approximated by projective ones. This criterion applies nicely to the case of smooth families\footnote{As recalled in Paragraph \ref{subs:jacob}, it is simply a holomorphic proper submersion whose fibres are complex tori.} of tori (this was already observed in \cite{CCE2}) and can even be used when the family is equivariant under the action of a finite group. 

\begin{thm}\label{th:th principal}
Let $f:X\to B$ be a smooth family of tori whose total space is compact Kähler and let us assume that $f$ is equivariant with respect to the action of a finite group $\Gamma$ on both $X$ and $B$. Then there exists a smooth family of tori of the form 
$$\mathcal{X}\stackrel{\pi}{\To} T\times B\stackrel{p_1}{\To} T$$
with $T$ a polydisk  and a point $t_0\in T$ such that the family $\mathcal{X}_{t_0}:=(\pi\circ p_T)^{-1}(t_0)$ is (isomorphic to) the initial one. This family has moreover the following properties:
\begin{enumerate}[$(i)$]
\item the group $\Gamma$ acts on $\mathcal{X}$,
\item the projection $\pi$ is equivariant with respect to this action on $\mathcal{X}$ and to the action on $T\times B$ induced by the trivial one on $T$,
\item the set $T_{alg}$ of points $t\in T$ such that $\mathcal{X}_t\to B$ has a multisection and its fibres are abelian varieties is dense near $t_0$.
\end{enumerate}
\end{thm}
The last sentence means that the closure\footnote{The set $T_{alg}$ is the set of parameters such that the corresponding deformation is \og as algebraic as possible\fg. It should be noted that $X_t$ is not necessarily projective (even when $t\in T_{alg}$) since $B$ is not assumed to be so.} of $T_{alg}$ contains an open neighbourhood of $t_0$. Up to shrinking $T$ we can assume that $\overline{T_{alg}}$ is thus the whole of $T$.

In particular, Theorem \ref{th:th principal} shows that the problem of approximating compact Kähler manifolds with projective ones has a positive answer in the case of smooth tori families.

\begin{cor}\label{cor: Kodaira famille tores}
Let $X$ be a compact Kähler manifold and let us assume that there is a finite étale Galois cover $\tilde{X}\to X$ which is the total space of a smooth family of tori over a projective base (equivariant under the action of the Galois group). Then $X$ can be approximated by projective manifolds: it is the central fibre of a smooth morphism $(\mathcal{X},X)\to (T,t_0)$ (with $T$ smooth) and the set of $t\in T$ such that $\mathcal{X}_t$ is projective is dense near $t_0$.
\end{cor}
\begin{proof}
We can apply Theorem \ref{th:th principal} to the smooth family of tori $f:\tilde{X}\to B$ and to action of $\Gamma:=\mathrm{Gal}(\tilde{X}/X)$. We get a smooth deformation $\tilde{\mathcal{X}}\to T\times B$ of the initial family (over $t_0$) and we can assume the set of points $t\in T$ such that $\tilde{\mathcal{X}}_t\To B$ has a multisection and its fibres are abelian varieties is dense in $T$. The manifolds $\tilde{\mathcal{X}}_t$ having these properties are thus projective according to \cite{C81}. Since the action of $\Gamma$ is free on $\tilde{\mathcal{X}}_{t_0}$ we can assume that it is also free on $\tilde{\mathcal{X}}$ (up to shrinking $T$). The family $\mathcal{X}:=\tilde{\mathcal{X}}/\Gamma\To T$ is thus a smooth deformation of $\mathcal{X}_{t_0}\simeq \tilde{X}/\Gamma\simeq X$ and the set of points $t\in T$ such that $\mathcal{X}_t$ is projective is dense in $T$ (the quotient of a projective manifold by a finite group is still projective).
\end{proof}

Theorem \ref{th:th principal} together with the structure results obtained in \cite{CCE1} yields a definitive answer to Question \ref{qt:kahler group} in the linear case.
\begin{cor}\label{cor:kahler linéaire}
A Kähler group which is linear is also a projective one: the fundamental group of a compact Kähler manifold can be realised as the fundamental group of a smooth projective variety if it is a linear group.
\end{cor}
\noindent Let us recall that the main result of \cite{CCE2} is a version of the latter corollary \emph{up to finite index}. It is stated there that a linear Kähler group has a finite index subgroup which is projective. In the sequel of this article, we will explain how to get rid of this finite index subgroup. Proofs of Theorem \ref{th:th principal} and Corollary \ref{cor:kahler linéaire} will be given in Paragraph \ref{subs:conclusion}.\\

Before presenting the ingredients involved in these proofs, let us give a word of explanation on the relative deformation constructed in Theorem \ref{th:th principal} (the reader is advised to consult \cite[\S 3.4.2, p. 191]{S06} for the notions concerning relative deformations). The infinitesimal relative deformations of a smooth morphism $f:X\to B$ are described by the space $\coh^1(X,T_{X\mid B})$ (\emph{cf.} Lemma 3.4.7 in \emph{loc. cit.})  
and the Leray spectral sequence for $T_{X\mid B}$ and $f$ gives a piece of exact sequence:
\begin{equation}\label{eq:Leary tangent relatif}
0\To \coh^1(B,f_*T_{X\mid B})\To \coh^1(X,T_{X\mid B})\To \coh^0(B,R^1f_*T_{X\mid B}).
\end{equation}
In our situation (smooth families of tori), both sides of (\ref{eq:Leary tangent relatif}) correspond to a different type of relative deformation. The left-hand side parametrizes relative deformation using translations in the fibres of $f$ (see the content of Proposition \ref{prop:deformation c constant}) whereas the right-hand side has to do with deformation of the variation of Hodge structures induced by $f$ (these deformations are identified in Paragraph \ref{subs:buchdahl}). In a sense, the strategy of the proof is thus dictated by the terms appearing in (\ref{eq:Leary tangent relatif}).

\section{Smooth families of tori}\label{sec:smooth family}
We recall here some basic facts about smooth families of tori: their description as torsors and their deformations. We then put this study in an equivariant framework. Some facts recalled in Paragraph \ref{subs:jacob} already appear in \cite[\S 2]{Nak99}. Our reference concerning Hodge theory is \cite{V02}. For more advanced material on Jacobian fibrations, the reader is referred to \cite{Sch,BP13} and to the references therein.

\subsection{Jacobian fibrations}\label{subs:jacob}
 
Let $f:X\to B$ be a proper submersion between complex manifolds. We assume moreover that the fibres of $f$ are complex tori. We shall call such a fibration a \emph{smooth family of tori}. This fibration determines\footnote{Since in the sequel we will have to change the Hodge structure keeping the local system fixed, we will use calligraphic letters when referring to a \vsh~and straight ones to denote the underlying local system.} a variation of Hodge structures $\cH$  (\vsh~for short) of weight $-1$ and rank $2g$ where $g:=\dim(f)$ is the relative dimension of $f$. Let us recall that, in this weight one situation, a \vsh~consists of an even rank local system $\HH_\ZZ$ and a holomorphic subbundle $\cF$ of $\mathcal{V}:= \HH_\ZZ\otimes \mathcal{O}_B$ satisfying the Hodge symmetry:
$$\mathcal{V}_b=\cF_b\oplus\bar{\cF}_b$$
for any $b\in B$. The \vsh~associated with a tori family is given by the following data: the underlying local system is
$$\HH_\ZZ:=\mathrm{Hom}(R^1f_*\ZZ_X,\ZZ_B),$$
the Hodge filtration being given by
$$\cF:=\mathrm{Hom}(R^1f_*\mathcal{O}_X,\mathcal{O}_B)\subset \HH_\ZZ\otimes \mathcal{O}_B.$$
Let us remark that the duality
$$R^1f_*\ZZ_X\otimes R^{2g-1}f_*\ZZ_X\To R^{2g}f_*\ZZ_X\simeq \ZZ_B$$
shows that $\HH_\ZZ$ is isomorphic to $R^{2g-1}f_*\ZZ_X$.

With these data we can associate a particular family of tori. Let us consider the injection
$$\HH_\ZZ\hookrightarrow \cE:=\HH_\ZZ\otimes \mathcal{O}_B/\cF\simeq f_*T_{X\mid B}$$
to this end. It can be used to define an action of $\HH_\ZZ$ on the total space of $\cE$ and the quotient gives rise to a smooth family of tori which will be denoted
$$p:J(\cH)\To B$$
and called the \emph{Jacobian fibration} associated with $\cH$. This fibration comes endowed with a natural section (the image of the zero section of $\cE$) and using it as the origins of the fibres we can define an abelian group law on the sections of $p$. We will denote by $\cJ(\cH)$ this sheaf of abelian groups which sits in the following short exact sequence:
\begin{equation}\label{eq:suite exacte sections}
0\To \HH_\ZZ\To \cE \To \cJ(\cH)\To 0.
\end{equation}

Let us say a word about polarizations (inspired from \cite[p. 15-17]{Nak99}). A real polarization of $\cH$ is a flat non degenerate skew-symmetric bilinear\footnote{Here and below, $\HH_\QQ$ and $\HH_\RR$ stand for $\HH_\ZZ\otimes\QQ$ and $\HH_\ZZ\otimes\RR$ as usual.} form
$$q:\HH_\RR\times \HH_\RR\To \RR_B$$
satisfying the Hodge-Riemann relations:
$$q(\cF,\cF)=0\quad \mathrm{and}\quad \forall\, 0\neq x\in \cF,\,\, i q(x,\bar{x})>0.$$
The polarization is said to be rational if it defined on $\HH_\QQ$ (with values in $\QQ_B$). If such a rational polarization exists, we shall say that $\cH$ is $\QQ$-polarizable. In this case, the corresponding tori are abelian varieties.

Once such a polarization is fixed, the period domain of $(\HH_\ZZ,q)$ can be identified with the Siegel half space
$$\mathbf{H}_g:=\left\{ \tau\in \mathrm{M}_g(\CC)\mid \tau^t=\tau\,\mathrm{and}\,\Im\mathrm{m}(\tau)>0 \right\}$$
and the representation associated with the local system has its value in the symplectic group
$$\pi_1(B)\To \mathrm{Sp}_g(\ZZ).$$
This can then be used to define an action of $\pi_1(B)\ltimes \ZZ^{2g}$ on $\tilde{B}\times \CC^g$, the resulting quotient being another realization of the Jacobian fibration. In this case, the Jacobian fibration is endowed with a relative Kähler form $\omega_q$: its restriction to any fiber is a Kähler metric. If $q$ is rational, the fibration $J(\cH)\to B$ is then a locally projective morphism.

In the reverse direction, starting from a smooth family of tori $f:X\to B$ inducing $\cH$, it is obvious that a (relative) Kähler metric $\omega$ on $X$ induces a real polarization $q_\omega$ on $\HH_\RR$.

\subsection{Smooth families of tori as torsors}

Now it is well known that the initial family $f:X\to B$ can be seen as a torsor under the Jacobian fibration and as such can be described by an element
$$\eta(f)\in \coh^1(B,\cJ(\cH)).$$
Here is a simple description of the class $\eta(f)$. If $(U_i)$ is an open cover of $B$ such that $f^{-1}(U_i)\to U_i$ has a section $\sigma_i$ then the quantity $\eta_{ij}:=\sigma_i-\sigma_j$ is a perfectly well defined cocycle with values in $\cJ(\cH)$. Conversely, given a cohomology class $\eta$ represented by a cocycle $(\eta_{ij})$, we can look at the isomorphisms induced by the sections $\eta_{ij}$ (translations in the fibres):
$$\tr(\eta_{ij}):p^{-1}(U_{ij})\stackrel{\sim}{\To}p^{-1}(U_{ij})$$
defined by the formulas:
$$\tr(\eta_{ij})(x)=x+\eta_{ij}(p(x))$$
(the addition referring to the one in $J(\cH)$). The isomorphisms $\tr(\eta_{ij})$ satisfy a cocyle relation and we can use them to glue the fibrations $p^{-1}(U_{i})\to U_i$ into a new family of tori $J(\cH)^\eta\to B$ (and both mechanisms are inverse one to each other).
\begin{prop}\label{prop:smooth family = torsor}
There is a one-to-one correspondence between isomorphism classes of smooth families of tori $f:X\to B$ inducing $\cH$ and the cohomology classes $\eta\in \coh^1(B,\cJ(\cH)).$ In particular, if $f:X\to B$ and $g:Y\to B$ are smooth families of tori inducing the same \vsh~on B, we can glue them over $B$ to get a new family $h:Z\to B$ such that $\eta(h)=\eta(f)+\eta(g)$.
\end{prop}

With this in mind it is obvious that there always exists an étale morphism
$$J(\cH)^\eta\to J(\cH)^{m\cdot \eta}$$
for $\eta\in \coh^1(B,\cJ(\cH))$ and $m\ge1$ an integer (obtained by gluing the multiplication by $m$ defined on the Jacobian fibration). In particular, if $\eta$ is torsion (of order $m$ say), $J(\cH)^\eta$ appears as a finite étale cover of $J(\cH)^{m\cdot \eta}=J(\cH)$ and, in that case, the pull-back of the canonical section of $J(\cH)/B$ gives rise to a multisection of $J(\cH)^\eta/B$ (which is étale over $B$ by its very construction). This proves at least one implication of the following proposition.
\begin{prop}\label{prop:eta torsion multisection}
Let $f:X\to B$ be a smooth family of tori (inducing the \vsh~$\cH$). The class $\eta(f)$ is torsion in $\coh^1(B,\cJ(\cH))$ if and only if $f$ has a multisection. If it is the case, the multisection can be chosen étale over $B$.
\end{prop}

\begin{rem}\label{rem:definition intrinseque eta}
Using relative Deligne groups (as in \cite[\S 2]{Nak99}), we can give an intrinsic definition of the class $\eta(f)$ associated with a family of tori $f:X\to B$. Let us look at the following complex:
\begin{equation}\label{eq:complex Deligne}
\ZZ_{\mathcal{D}}^\bullet(X/B)\,:\, 0\To \ZZ_X\To \mathcal{O}_X \stackrel{\drel}{\To} \Omega^1_{X\mid B} \stackrel{\drel}{\To} \dots \stackrel{\drel}{\To} \Omega^{g-1}_{X\mid B}
\end{equation}
where $\drel$ denotes the relative differential. The complex (\ref{eq:complex Deligne}) sits obviously in the exact sequence
\begin{equation}\label{eq:suite complexe Deligne}
0\To \Omega^{\leq g-1}_{X\mid B}[-1]\To \ZZ_{\mathcal{D}}^\bullet(X/B)\To \ZZ_X\To 0
\end{equation}
where the last term is the complex given by the constant sheaf concentrated in degree 0. Taking derived direct image of (\ref{eq:suite complexe Deligne}) yields a triangle:
\begin{equation}\label{eq:suite exacte derivee}
\RR f_*\Omega^{\leq g-1}_{X\mid B}[-1]\To \RR f_*\ZZ_{\mathcal{D}}^\bullet(X/B)\To \RR f_*\ZZ_X \stackrel{+1}{\To} .
\end{equation}
On the other hand, we also have another triangle:
\begin{equation}\label{eq:suite derivee Hodge}
\RR f_*\Omega^{\geq g}_{X\mid B}\To \RR f_*\Omega^{\bullet}_{X\mid B}\To \RR f_*\Omega^{\leq g-1}_{X\mid B}\stackrel{+1}{\To} 
\end{equation}
and the long exact sequence of cohomology associated with (\ref{eq:suite derivee Hodge}) shows that
\begin{equation}\label{eq:iso derivee}
\coh^k \RR f_*\Omega^{\leq g-1}_{X\mid B}\simeq \left(R^kf_*\CC_X\otimes \mathcal{O}_B\right)/F^g
\end{equation}
where $F^g$ is the $g^{\mathrm{th}}$-step of the Hodge filtration on the \vsh~$R^kf_*\CC_X$. Now looking at the long exact sequence associated with (\ref{eq:suite exacte derivee}), we get:
$$R^{2g-1} f_*\ZZ_X\to \coh^{2g-1} \RR f_*\Omega^{\leq g-1}_{X\mid B} \to \coh^{2g} \RR f_*\ZZ_{\mathcal{D}}^{\bullet}(X/B)\to R^{2g}f_*\ZZ_X\to 0.$$
We can identify several terms in the sequence above: $R^{2g}f_*\ZZ_X$ is the constant sheaf $\ZZ_B$ and $R^{2g-1} f_*\ZZ_X$ is nothing but $\HH_\ZZ$. Using the isomorphism (\ref{eq:iso derivee}), the last piece of exact sequence reads as:
\begin{equation}\label{eq:extension Deligne}
0\To \cJ(\cH)\To \mathcal{D}_0(X/B):=\coh^{2g} \RR f_*\ZZ_{\mathcal{D}}^{\bullet}(X/B)\To \ZZ_B\To 0
\end{equation}
which is nothing but a relative version of \cite[cor. 12.27, p. 285]{V02}. So we have just associated with $f:X\to B$ an extension of the sheaf $\cJ(\cH)$ by $\ZZ_B$ and it is fairly clear that the cohomology class $\eta(f)$ is obtained as the image of $1$ under the connecting morphism
$$\delta_f:\coh^0(B,\ZZ_B)\To \coh^1(B,\cJ(\cH))$$
coming from (\ref{eq:extension Deligne}).

As the name suggests, the sheaf $\mathcal{D}_0(X/B)$ should be thought as a sheaf of relative 0-cycles of $X/B$. With this in mind, we see that a multisection of $f$ determines a global section of $\mathcal{D}_0(X/B)$ which is sent to some non zero integer in $\coh^0(B,\ZZ_B)$ (the relative degree of the corresponding cycle) and the description of $\eta(f)$ we got above implies that this class should be a torsion one, thus proving the second implication in Proposition \ref{prop:eta torsion multisection}.
\end{rem}

Now we can use the exact sequence (\ref{eq:suite exacte sections}) to define a topological invariant of a smooth family of tori. The long exact sequence associated with (\ref{eq:suite exacte sections}) reads as
$$\coh^1(B,\cE)\stackrel{\exp}{\To} \coh^1(B,\cJ(\cH))\stackrel{\cc}{\To} \coh^2(B,\HH_\ZZ).$$
The following was first observed by Kodaira in his study of elliptic surfaces \cite[Theorem 11.3]{Kod}.
\begin{prop}\label{prop:deformation c constant}
Let us fix a class $\eta_0$ in $\coh^1(B,\cJ(\cH))$. Then any finite dimensional vector space $V\subset\coh^1(B,\cE)$ appears as the base space of a smooth deformation
$$\pi:\mathcal{X}^{\eta_0}_V\to V\times B$$
such that if $v\in V$ the smooth family of tori
$$\pi_v:\mathcal{X}^{\eta_0}_v:=\pi^{-1}(\{v\}\times B)\to B$$
is such that $\eta(\pi_v)=\exp(v)+\eta_0$.

In particular, if $\cc(\eta_0)$ is torsion, $J(\cH)^{\eta_0}$ can be deformed (over $B$) to a smooth family of tori having a multisection.
\end{prop}
\begin{proof}
There is a tautological vector bundle $\cE_V$ which is an extension:
$$0\To \cE\To \cE_V\To \underline{V}\To 0$$
where $\underline{V}$ is the trivial vector bundle. Its extension class is given by
$$\mathrm{Id}_V\in\mathrm{End}(V)\subset V^*\otimes\coh^1(B,\cE)\simeq\coh^1(B,\underline{V}^*\otimes\cE).$$
The local system $\HH_\ZZ$ acts on the total space of $\cE_V$ by translations and we can form the quotient. The manifold $\mathcal{Y}$ we obtain has a natural projection to the total space of $\underline{V}$. This is thus a smooth family of tori
$$\rho:\mathcal{Y}\To V\times B$$
and over a point $v\in V$ we get from the construction that $\eta(\rho_v):=\exp(v)$. Now we can glue the trivial family $V\times X\to V\times B$ and $\mathcal{Y}\to V\times B$ over $V\times B$ to get the sought family
$$\pi:\mathcal{X}\To V\times B.$$

If $\cc(\eta_0)$ is torsion then there exists an integer $m\ge1$ such that $m\cdot \eta_0=\exp(v_0)$ for some $v_0\in \coh^1(B,\cE)$. The latter being a vector space we can rewrite this equality as $m\cdot (\eta-\exp(v_0/m))=0$. The construction explained above with $\CC\cdot v_0\subset \coh^1(B,\cE)$ gives a smooth family of tori $\mathcal{X}\to \CC\times B$ such that $\eta(\mathcal{X}_0)=\eta_0$ and $\eta(\mathcal{X}_{-1/m})$ is torsion. The family $\mathcal{X}_{-1/m}\to B$ has thus a multisection according to Proposition \ref{prop:eta torsion multisection}.
\end{proof}

Let us remark that the situation corresponding to the second part of the preceding proposition occurs in the Kähler case\footnote{Let us note that if $f:X\to B$ is proper and smooth, and if $X$ is Kähler, then $B$ is Kähler as well. If $d$ is the relative dimension of $f$ and $\omega$ a Kähler form on $X$, the fibrewise integration of $\omega^{d+1}$ provides us with a Kähler form on $B$.}.
\begin{prop}\label{prop:kahler implique c torsion}
Let $f:X\to B$ be a smooth family of tori inducing $\cH$. If $X$ is Kähler, the class $\cc(\eta(f))$ is torsion in $\coh^2(B,\HH_\ZZ)$.
\end{prop}
\begin{proof}
Using the description of the class $\eta(f)$ given in Remark \ref{rem:definition intrinseque eta}, we readily infer that there is a commutative diagram
\begin{equation}\label{eq:diag Leray}
\xymatrix{\coh^0(B,\ZZ_B=R^{2g}f_*\ZZ_X)\ar[rd]_{\delta_f}\ar[r]^{d_2} & \coh^2(B,R^{2g-1}f_*\ZZ_X=\HH_\ZZ) \\
& \coh^1(B,\cJ(\cH))\ar[u]_{\cc}}
\end{equation}
where $d_2$ is the differential appearing in the Leray spectral sequence associated with $f$ and $\ZZ_X$. But it is well known that this spectral sequence degenerates at $E_2$ for a Kähler morphism and when it is computed using real coefficients \cite[Prop. 2.4]{D68} (see also \cite[Th. 16.15, p. 379]{V02}). The diagram (\ref{eq:diag Leray}) is translated into the equality
$$\cc(\eta(f))=\cc(\delta_f(1))=d_2(1)$$
and the vanishing of $d_{2,\RR}$ exactly means that $\cc(\eta(f))$ is torsion.
\end{proof}

\begin{rem}\label{rem:fibration relativement kahler}
Obviously a relative Kähler class is enough to get the same conclusion as above. It is quite surprising that $\cc(\eta)$ being torsion is in fact equivalent to the fibration $J(\cH)^\eta\to B$ being cohomologically Kähler (meaning that there is a class of degree 2 on $J(\cH)^\eta$ whose restriction to the fibres is a Kähler class). This is the content of \cite[Proposition 2.17]{Nak99}.
\end{rem}

\subsection{Equivariant cohomology}\label{subs:G-cohomologie}

In this paragraph, we recall some facts about equivariant cohomology with respect to the action of a (finite) group $\Gamma$. This formalism was also used in the study of elliptic surfaces \cite[\S 13-14]{Kod}. Here is the setting: we consider a finite group $\Gamma$ acting on a complex manifold $B$ and we look at sheaves of abelian groups $\cF$ over $B$ endowed\footnote{This is equivalent to giving an action on the étalé space $\mathbb{F}$ associated with $\cF$ such that the natural projection $\mathbb{F}\to B$ is $\Gamma$-equivariant.} with an action of $\Gamma$ compatible with the one on $B$: it means that for any $\gamma\in \Gamma$, there exists an isomorphism
$$i_\gamma:\gamma_*\cF\stackrel{\sim}{\To}\cF$$
or, even more concretely, for any open subset $U\subset B$, there is an isomorphism
$$i_\gamma:\coh^0(U,\cF)\stackrel{\sim}{\To} \coh^0(\gamma^{-1}(U),\cF).$$
The collection of these isomorphisms has to satisfy the cocycle relation:
$$i_{\gamma g}=i_\gamma\circ (\gamma_* i_g).$$
If $\cF$ is such a $\Gamma$-sheaf, the group $\Gamma$ acts on the space of global sections and we can define the following functor:
$$\mathrm{F}_\Gamma:\left\{\begin{array}{ccc}\mathcal{S}h_\Gamma(B) & \To & \mathcal{A}b\\ \cF & \mapsto & \coh^0(B,\cF)^{\Gamma} \end{array}\right.$$
from the category of $\Gamma$-sheaves (of abelian groups) to the category of abelian groups.

\begin{defi}\label{def:coh equivariant} 
The equivariant cohomology groups of a $\Gamma$-sheaf $\cF$ are defined using the (right) derived functors of $\mathrm{F}_\Gamma$:
$$\coh^i_\Gamma(B,\cF):=R^i\mathrm{F}_\Gamma(\cF).$$
\end{defi}

The functor $\mathrm{F}_\Gamma$ being expressed as the composition of two functors (taking first the global sections and then the invariants under $\Gamma$), the equivariant cohomology groups can be computed using the spectral sequence of a composed functor (see \cite[th. 16.9, p. 371]{V02}).

\begin{prop}\label{prop:suite spectral coh equi}
For any $\Gamma$-sheaf $\cF$, there is a spectral sequence
\begin{equation}\label{eq:sectral seq}
E_2^{p,q}:=\coh^p(\Gamma,\coh^q(B,\cF))\Longrightarrow \coh^{p+q}_\Gamma(B,\cF)
\end{equation}
abutting to the equivariant cohomology of $\cF$.
\end{prop}

\begin{rem}\label{rem:cohomology equi coherent}
It is well known that the higher cohomology groups $\coh^p(\Gamma,M)$ are torsion groups for any $\Gamma$-module $M$ and for any $p>0$ when $\Gamma$ is finite (see \cite[Chapter III, Corollary 10.2]{Br}). In particular, if $M$ is in addition a vector space, then the groups $\coh^p(\Gamma,M)$ vanish for $p>0$. It applies for instance when $M=\coh^q(B,\cF)$ for $\cF$ a $\Gamma$-sheaf which is at the same time a coherent sheaf. In this case, the spectral sequence from the preceding proposition degenerates and the equivariant cohomology consists in nothing but taking the invariants:
$$\coh^{i}_\Gamma(B,\cF)=\coh^i(B,\cF)^\Gamma.$$
\end{rem}

\subsection{Smooth family of tori endowed with a group action}\label{subs:fibration equivariante}

We now aim at applying results from the previous paragraph to the following situation: $f:X\to B$ is smooth family of tori endowed with an action of a finite group $\Gamma$. The fibration $f$ is equivariant with respect to both actions of $\Gamma$ on $X$ and $B$. In particular, all the natural objects arising in this situation (the local system $\HH_\ZZ$, the \vsh, the Jacobian fibration as well as its sheaf of sections) are endowed with compatible actions of $\Gamma$. The sequence (\ref{eq:suite exacte sections}) is then an exact sequence of $\Gamma$-sheaves and using Remark \ref{rem:cohomology equi coherent} the long exact sequence reads now as:
\begin{equation}\label{eq:suite longue G-cohomologie}
\coh^1(B,\cE)^\Gamma\stackrel{\exp}{\To} \coh^1_\Gamma(B,\cJ(\cH))\stackrel{\cc_\Gamma}{\To} \coh^2_\Gamma(B,\HH_\ZZ)\dots
\end{equation}
As in Paragraph \ref{subs:jacob}, we can naturally identify a $\Gamma$-equivariant smooth family of tori $f:X\to B$ with its cohomology class
$$\eta_\Gamma(f)\in \coh^1_\Gamma(B,\cJ(\cH)).$$
This can be done as in Kodaira's work \cite[Theorem 14.1]{Kod} or using relative Deligne groups. The exact sequence
$$0\To \cJ(\cH)\To \mathcal{D}_0(X/B) \To \ZZ_B\To 0$$
is indeed an exact sequence of $\Gamma$-sheaves and the connecting morphism
$$\delta_f^\Gamma:\coh^0(B,\ZZ_B)^\Gamma=\ZZ\To \coh^1_\Gamma(B,\cJ(\cH))$$
enables us to define $\eta_\Gamma(f):=\delta_f^\Gamma(1)$ in the group $\coh^1_\Gamma(B,\cJ(\cH))$.\\

We can now turn Propositions \ref{prop:deformation c constant} and \ref{prop:kahler implique c torsion} into $\Gamma$-equivariant statements. The proof of Proposition \ref{prop:deformation c constant} applies verbatim to give the following result.

\begin{prop}\label{prop:def c constant equivariant}
Let us fix a class $\eta$ in $\coh^1_\Gamma(B,\cJ(\cH))$. Then any finite dimensional vector space $V\subset\coh^1(B,\cE)^\Gamma$ appears then as the base space of a smooth $\Gamma$-equivariant deformation
$$\pi:\mathcal{X}^\eta_V\to V\times B.$$
Precisely: the group $\Gamma$ acts on $\mathcal{X}^\eta_V$ and the morphism $\pi$ is equivariant for the trivial action of $\Gamma$ on $V$. If $v\in V$ the smooth family of tori
$$\pi_v:\mathcal{X}^{\eta}_v:=\pi^{-1}(\{v\}\times B)\to B$$
has the following cohomology class
$$\eta_\Gamma(\pi_v)=\exp(v)+\eta\in \coh^1_\Gamma(B,\cJ(\cH)).$$
\end{prop}

\begin{prop}\label{prop:deformation kahler equivariante}
Let $f:X\to B$ be a $\Gamma$-equivariant smooth family of tori and let us assume that $X$ is Kähler. Then the class
$$\cc_\Gamma(\eta_\Gamma(f))\in \coh^2_\Gamma(B,\HH_\ZZ)$$
is torsion and $f:X\to B$ can be deformed (over $B$) to another smooth family of tori having a multisection and acting on by $\Gamma$.
\end{prop}
\begin{proof}
Since the $E^{0,2}_\infty$ coming from the spectral sequence (\ref{eq:sectral seq}) is a subgroup of $E^{0,2}_2$, we have a natural morphism:
$$\coh^2_\Gamma(B,\HH_\ZZ)\stackrel{\pi^{0,2}}{\To} \coh^2(B,\HH_\ZZ)^\Gamma.$$
The following relation is clear:
$$\pi^{0,2}(\cc_\Gamma(\eta_\Gamma(f)))=\cc(\eta(f))$$
and consists in ignoring the $\Gamma$-action. Now we can use Proposition \ref{prop:kahler implique c torsion} to infer that $\pi^{0,2}(\cc_\Gamma(\eta_\Gamma(f)))$ is torsion. Finally the kernel of $\pi^{0,2}$ is an extension of $E^{2,0}_\infty$ by $E^{1,1}_\infty$ and these groups are torsion according to Remark \ref{rem:cohomology equi coherent}. It is enough to conclude that $\cc_\Gamma(\eta_\Gamma(f))$ is a torsion class in $\coh^2_\Gamma(B,\HH_\ZZ)$.

Since $\cc_\Gamma(\eta_\Gamma(f))$ is torsion, we can mimic the proof of Proposition \ref{prop:deformation c constant}: it produces a deformation
$$\mathcal{X}\To \CC\times B\To \CC$$
endowed with an action of $\Gamma$, the group acting fibrewise over $\CC$. Moreover there is a point in the base space $t\in \CC$ such that
$$\eta_\Gamma(\mathcal{X}_t\to B)\in \coh^1_\Gamma(B,\cJ(\cH))$$
is torsion and it implies that $\mathcal{X}_t\to B$ has a multisection (look at the natural projection $\coh^1_\Gamma(B,\cJ(\cH))\To \coh^1(B,\cJ(\cH))^\Gamma$).
\end{proof}

\section{From Kähler fibrations to projective ones}\label{sec:Kodaira pb}
\subsection{Deforming the \vsh}\label{subs:def vsh}

In this section we show how to deform a smooth family of tori once a deformation of the \vsh~is fixed. Let us make this more precise. We consider $f:X\to B$ a smooth family of tori between compact Kähler manifolds and as before we denote by $\cH$ the \vsh~induced on the local system $\HH_\ZZ$. We aim at considering small deformation of $\cH$ in the following sense.
\begin{defi}\label{defi:family of vsh}
A small deformation of $\cH$ is a \vsh~$\cH_U$ on $\HH_\ZZ$ seen as a local system on $U\times B$ where $U$ is a polydisk around $o\in U$ and such that the restriction of $\cH_U$ to $\{ o\}\times B\simeq B$ is the given $\cH$. We shall denote by $\cE_U$ the holomorphic vector bundle $\cH_U/\cH_U^{1,0}$.
\end{defi}

We will make use of the following lemma in the sequel.
\begin{lem}\label{lem:surjectivite Hodge}
Let $\mathbb{V}_\RR$ be a local system underlying a \vsh~$\mathcal{V}$ of weight $w$ defined on a compact Kähler manifold $B$. Then for any $k\ge 0$, the natural map
$$\coh^k(B,\mathbb{V}_\RR)\To \coh^k(B,\mathcal{V}/F^1)$$
induced by $\mathbb{V}_\RR\to \mathbb{V}_\CC\to \mathcal{V}\to\mathcal{V}/F^1$
 is surjective.
\end{lem}
\begin{proof}
The vector space $\coh^k(B,\mathbb{V}_\RR)$ carries a natural Hodge structure of weight $k+w$. This is Deligne's construction explained in \cite[Theorem 2.9]{Zuc} (see also \cite[\S 4.3]{thesedamien}). From the construction itself, the $(P,Q)$ part of this Hodge structure is given by the hypercohomology of a certain complex
$$\coh^k(B,\mathbb{V}_\CC)^{P,Q}=\HH^k(K^\bullet_{P,Q}).$$
It happens that when $(P,Q)=(0,k+w)$ this complex reduces to the Dolbeault complex
$$K^\bullet_{0,k+w}=\mathcal{A}^{0,\bullet}(\mathcal{V}^{0,w})$$
and its hypercohomology is thus the usual one of the holomorphic vector bundle $\mathcal{V}^{0,w}=\mathcal{V}/F^1$. The $(0,k+w)$ part of this Hodge structure is then given by
$$\coh^k(B,\mathbb{V}_\CC)^{0,k+w}\simeq \coh^k(B,\mathcal{V}/F^1).$$
Now it is an easy observation that the real vector space underlying a weight $n$ Hodge structure always surjects onto its $(0,n)$ Hodge component.
\end{proof}

With Definition \ref{defi:family of vsh} at hand, we have the following deformation process.
\begin{prop}\label{prop:deformation induite par vsh}
Let $f:X\to B$ be a smooth family of tori between compact Kähler manifolds inducing $\cH$ and $\cH_U$ a small deformation of $\cH$. Then there exists
$$\mathcal{X}_U \stackrel{\pi_U}{\To} U\times B\stackrel{p_1}{\To} U$$
a smooth family of tori over $U\times B$ inducing $\cH_U$ and such that the family of tori $(\pi_U\circ p_1)^{-1}(o)\to B$ is isomorphic to $X\to B$.
\end{prop}

\begin{proof}
The \vsh~$\cH_U$ being fixed we can consider the Jacobian fibration
$$\cJ(\cH_U)\to U\times B$$
associated with it and the corresponding long exact sequences:
\begin{equation}\label{eq:extending the deformation}
\xymatrix{\coh^1(U\times B,\cE_U)\ar[r]^{\exp}\ar[d] & \coh^1(U\times B,\cJ(\cH_U))\ar[r]^{\cc}\ar[d] & \coh^2(U\times B,\HH_\ZZ)\ar@{=}[d]\ar[r] &  \coh^2(U\times B,\cE_U)\\
\coh^1(B,\cE)\ar[r]^{\exp} & \coh^1(B,\cJ(\cH))\ar[r]^{\cc}& \coh^2(B,\HH_\ZZ) & }
\end{equation}
The vertical arrows in the preceding diagram are induced by the restriction to $B\simeq \{0\}\times B$. Since $\cc(\eta(f))$ is torsion (Proposition \ref{prop:kahler implique c torsion}), its image in the vector space $\coh^2(U\times B,\cE_U)$ vanishes and thus there exists a class $\eta^1\in \coh^1(U\times B,\cJ(\cH_U))$ whose restriction to $B$ satisfies $\cc(\eta^1_{\mid B})=\cc(\eta(f))$. It means that there exists a class $\alpha\in \coh^1(B,\cE)$ such that $\eta^1_{\mid B}-\eta(f)=\exp(\alpha)$. To conclude it is enough to observe that the first vertical arrow is surjective. To do so, let us consider the following diagram:
$$\xymatrix{\coh^1(U\times B,\mathbb{H}_\RR)\ar@{=}[d]\ar[r] & \coh^1(U\times B,\cE_U)\ar@{->>}[d] \\
\coh^1(B,\mathbb{H}_\RR) \ar@{->>}[r] & \coh^1(B,\cE).
}$$
Since the horizontal bottom arrow is surjective (Lemma \ref{lem:surjectivite Hodge}), it is then clear that the map we are interested in
$$\coh^1(U\times B,\cE_U)\To \coh^1(B,\cE)$$
is surjective as well. Now if $\alpha_U\in \coh^1(U\times B,\cE_U)$ is such that $(\alpha_U)_{\mid B}=\alpha$ then the class $\eta_U:=\eta^1-\exp(\alpha_U)$ restricts to $B$ as the given $\eta(f)$. The class $\eta_U$ corresponds thus to a smooth family of tori over $U\times B$ inducing $\cH_U$ and whose restriction to $B$ is isomorphic to the fibration $f:X\to B$ we started with.
\end{proof}

\begin{rem}\label{rem:def G-equivariant vsh}
The last proposition holds also in the equivariant setting (we wrote down the proof without a group acting to keep the notation readable). It is enough to use equivariant cohomology and it gives the following conclusion (let us recall that a $\Gamma$-\vsh~is a \vsh~such that the underlying local system is endowed with an action of $\Gamma$, the Hodge filtration being compatible with this action).
\end{rem}

\begin{prop}\label{prop:def G-equivariant vsh}
Let $f:X\to B$ be a smooth family of tori (between compact K\"ahler manifolds) equivariant under the action of a finite group $\Gamma$ on both $X$ and $B$. Let us moreover consider $\cH_U$ a small deformation of $\cH$ which at the same time a $\Gamma$-\vsh~for the action on $U\times B$ given by the trivial one on $U$. There exists then
$$\mathcal{X}_U \stackrel{\pi_U}{\To} U\times B\stackrel{p_1}{\To} U$$
a smooth family of tori over $U\times B$ as in Proposition \ref{prop:deformation induite par vsh} such that $\Gamma$ is acting on $\mathcal{X}_U$ and $\pi_U$ is equivariant for the trivial action on $U$. 
\end{prop}

\subsection{Buchdahl's criterion for families of tori}\label{subs:buchdahl}
We now recall the relative Buchdahl criterion we obtained in \cite[Th. 1.1]{CCE2} and explain how to make it equivariant (adapting Graf's arguments from \cite[\S 9]{Graf}).
 
\begin{prop}\label{prop:CCE2}
Let $\cH$ be a weight $-1$ and rank $2g$ \vsh~over $B$ (whose underlying local system is denoted $\HH_\ZZ$). Let us assume moreover that $\cH$ admits a real polarization $q$. Then there exists a small deformation $\cH_V$ of $\cH$ such that the set
$$\left\{v\in V\mid \cH_{V,v}\,\mathrm{admits\, a\, rational\, polarization} \right\}$$
is dense near $\mathrm{o}\in V$: its closure contains an open neighbourhood of \emph{o} (the notation $\cH_{V,v}$ is simply the restriction of $\cH_V$ to $\set{v}\times B\simeq B$).
\end{prop}
Since we need to check that the construction of \cite{CCE2} can be made in an equivariant framework, we recall how the proof goes.
\begin{proof}[Sketch of proof]
Let us consider the $\RR$-algebra
$$A_\RR:=\coh^0(B,\mathrm{End}(\HH_\RR)).$$
The \vsh~$\cH$ is nothing but an element $I\in A_\RR$ such that $I^2=-1$ and as such determines a complex structure on $A_\RR$. This structure can be enriched as follows. Let us consider the decomposition
$$A_\RR=A^I\oplus A^{-I}$$
where $A^I$ (\emph{resp.} $A^{-I}$) consists in elements of $A_\RR$ commuting with $I$ (\emph{resp.} anti-commuting with $I$). Multiplication by $I$ respects the decomposition and thus induces a complex structure on each piece. If we let
$$A^{-I}_\CC:=A^{1,-1}_\CC\oplus A^{-1,1}_\CC,$$
we then have a weight 0 Hodge structure on $A_\RR$ whose $(0,0)$ part is just $A^I$.

Let $G$ be the group of invertible elements of $A_\RR$: it acts on $A_\RR$ by conjugation. The orbit through $I$ is $G/G^\circ$ where $G^\circ$ is the group of invertible elements commuting with $I$. The space $G/G^\circ$ inherits a complex structure from the local diffeomorphism
$$G/G^\circ\To G_\CC/G^\circ_\CC.$$
Let us consider a small neighbourhood $V$ of $o$ the class of the identity in $G/G^\circ$: it is the base of a tautological family of complex structures on $\HH_\RR$, \emph{i.e.} it carries a small deformation $\cH_V$ of $\cH$. Now we can consider the following\footnote{Here a remark is in order. Usually to be able to endow the cohomology of a \vsh~with a Hodge structure, the base manifold needs to be compact Kähler or at least a Zariski open subset of a compact Kähler manifold (in the latter case we end up with a mixed Hodge structure). But in our situation we only have to handle Hodge structures (on the global sections) coming from weight one \vsh. Since the complex structure $I$ commutes with the monodromy of the underlying local system, the Hodge decompositions induced on tensor products are compatible with the action of the monodromy group and these decompositions are preserved when taking the invariants. That is the reason why no assumption is needed on $B$ in our study.} weight 2 \vsh~on $V$: the local system is given by
$$\mathbb{W}_\QQ:=\coh^0(B,\Lambda^2\HH_\QQ^{\scriptscriptstyle{\vee}})$$
and the Hodge structure on $\mathcal{W}_v$ is induced by $\cH_v$ for $v\in V$. Now we aim at applying \cite[Proposition 17.20]{V02} and we first remark that the polarization can be seen as an element $q\in \mathcal{W}_o^{1,1}$. Moreover such an element induces in particular a morphism (of bidegree $(1,1)$) of Hodge structures
$$q\circ\cdot :A_\RR \To \mathcal{W}_{o}$$
which is clearly surjective ($q$ is an isomorphism between $\HH$ and $\HH^{\scriptscriptstyle{\vee}}$). It implies that the following component of the differential of the period map
\begin{equation}\label{eq:nabla Hodge}
\bar{\nabla}_o(q):T_{V,o}=A^{-I}=A_\RR\cap\left(A_\CC^{1,-1}\oplus A_\CC^{-1,1}\right)\To \mathcal{W}_o^{0,2}
\end{equation}
is surjective. The statement of \cite[Proposition 17.20]{V02} ensures that the set of $v\in V$ such that $\cH_v$ is $\QQ$-polarizable is dense in $V$.
\end{proof}

\noindent From the proof we get the following equivariant version.
\begin{cor}\label{cor:def vsh equivariant}
Let us assume that a finite group $\Gamma$ acts on $B$ and that the \vsh~$\cH$ is a $\Gamma$-\vsh. Then there exists a small deformation $\cH_U$ of $\cH$ on $U\times B$ which is at the same time a $\Gamma$-\vsh~over $U\times B$ for the trivial action on $U$ and such that the set of points $u\in U$ corresponding to $\QQ$-polarizable complex structures is dense in $U$.
\end{cor}
\begin{proof}
Let us consider the small deformation $\cH_V$ constructed in the above proof. It is obvious from the construction that $\Gamma$ acts on $V$ and that $\cH_V$ is a $\Gamma$-\vsh~for the diagonal action of $\Gamma$ on $V\times B$. Now let us restrict it to the set $U:=V^\Gamma$ of fixed points of $\Gamma$ in $V$. Since we saw that the space $V$ can be identified with an open neighbourhood of $0\in A_\RR\cap\left(A_\CC^{1,-1}\oplus A_\CC^{-1,1}\right)$ and since $\Gamma$ acts linearly on the latter vector space, we see that $U$ is smooth\footnote{It is a general fact: the set of fixed points $X^\Gamma$ of a finite group acting on a complex manifold is smooth, see \cite{Car}.} near the point $o$. Replacing the polarization $q$ with its average over the group $\Gamma$ we can assume that $q$ is $\Gamma$-invariant. Finally, since we are dealing with vector spaces, taking the invariants under the group $\Gamma$ preserves surjectivity in (\ref{eq:nabla Hodge}):
$$\bar{\nabla}_o(q)^\Gamma:T_{V,o}^\Gamma=T_{U,o}\To \left(\mathcal{W}_o^{0,2}\right)^\Gamma.$$
The use of \cite[Proposition 17.20]{V02} in this invariant context (we apply it to the \vsh~$\mathcal{W}^\Gamma$) shows that we can endow $\cH_U:=(\cH_V)_{\mid U}$ with a $\Gamma$-invariant polarization $q_U$ such that $q_u$ is a rational polarization of $\cH_u$ for a dense set of points $u\in U$.
\end{proof}

\subsection{Proofs of main statements}\label{subs:conclusion}

We are now in position to prove the main statements of this article.
\begin{proof}[Proof of Theorem \ref{th:th principal}]
Let $f:X\to B$ be a smooth family of tori with $X$ compact Kähler and assumed to be equivariant under the action of a finite group $\Gamma$. We denote by $\cH$ the \vsh~induced on the local system $\HH_\ZZ$. We first apply Corollary \ref{cor:def vsh equivariant}: it produces a small deformation $\cH_U$ (over a polydisk $U$) of $\cH$ as a $\Gamma$-\vsh~and such that the \vsh~$\cH_u$ is $\QQ$-polarizable for a dense subset of $U$. Since $X$ is compact Kähler, we can apply Proposition \ref{prop:def G-equivariant vsh}: there exists a smooth family of tori $f_U:\mathcal{X}_U\To U\times B$ inducing $\cH_U$ and such that $\Gamma$ acts equivariantly on $\mathcal{X}_U\To U\times B$ (with the trivial action on $U$). 

Now we need to find the right space of deformation to get the density statement. First we apply Lemma \ref{lem:surjectivite Hodge} to infer that
$$\coh^1(B,\mathbb{H}_\RR)\To \coh^1(B,\cE)$$
is surjective. Since we want to use $\coh^1(B,\cE)^\Gamma$ as a space of deformation for the family $f_U$, we look at the following commutative diagram:
$$\xymatrix{\coh^1(U\times B,\mathbb{H}_\RR)\ar@{=}[d]\ar[r] & \coh^1(U\times B,\cE_U)\ar@{->>}[d] \\
\coh^1(B,\mathbb{H}_\RR) \ar@{->>}[r] & \coh^1(B,\cE).
}$$
We can remark that taking the invariants yields a diagram of the same shape:
$$\xymatrix{\coh^1(U\times B,\mathbb{H}_\RR)^\Gamma\ar@{=}[d]\ar[r] & \coh^1(U\times B,\cE_U)^\Gamma\ar@{->>}[d] \\
\coh^1(B,\mathbb{H}_\RR)^\Gamma \ar@{->>}[r] & \coh^1(B,\cE)^\Gamma.
}$$
Let us consider $V$ the image of $\coh^1(B,\mathbb{H}_\RR)^\Gamma$ in $\coh^1(U\times B,\cE_U)^\Gamma$ and similarly $V_\QQ\subset V$ the image of $\coh^1(B,\mathbb{H}_\QQ)^\Gamma$ in $\coh^1(U\times B,\cE_U)^\Gamma$. Let us remark that the subset $V_\QQ$ is obviously dense in $V$. We can use $V$ in Proposition \ref{prop:def c constant equivariant} to construct a $\Gamma$-equivariant deformation:
$$f_{U,V}:\mathcal{X}_{V,U}\To V\times U\times B$$
such that $\mathcal{X}_{U,0}\To \{0\}\times U\times B$ is the previous $f_U$. Moreover the points of $V_\QQ$ are sent to torsion points in
$$\coh^1_\Gamma(U\times B,\cJ(\cH_U))$$
and thus to smooth families of tori with multisections. Finally if we denote by $U_{alg}$ the set of points $u\in U$ such that $\cH_u$ is $\QQ$-polarizable, the set
$$T_{alg}:=V_\QQ\times U_{alg}\subset T:=V\times U$$
is dense in $T$ and parametrizes families $\mathcal{X}_{v,u}$ having multisections and abelian varieties as fibres.
\end{proof}

Before giving the proof of Corollary \ref{cor:kahler linéaire}, let us state the structure result obtained in \cite{CCE1}. This is the main ingredient in the above mentioned proof.

\begin{thm}\label{th:structure CCE}
Let $X$ be a compact Kähler manifold and $\rho:\pi_1(X)\to \mathrm{GL}_N(\CC)$ be a linear representation. If $H<\pi_1(X)$ is a finite index subgroup such that $\rho(H)$ is torsion free, then the étale cover $Y\to X$ corresponding to $H$ has the following property (up to bimeromorphic transformations): the base of the Shafarevich morphism $sh_\rho:Y\to Sh_\rho(Y):=W$ is such that the Iitaka fibration of $W$ is a smooth family of tori $f:W\to B$ (onto a projective manifold of general type).
\end{thm}
We refer to \emph{loc. cit.} for the relevant notions. We will also make use of the following lemma.
\begin{lem}\label{lem:gp extension}
Let $\Gamma$ be a finite group acting on a topological space $X$. If $\mathbb{P}$ is any simply connected space endowed with a free action of $\Gamma$, the finite étale cover $X\times \mathbb{P}\to (X\times \mathbb{P})/\Gamma$  gives rise to an exact sequence of fundamental groups:
$$1\To \pi_1(X)\To \pi_1\left((X\times \mathbb{P})/\Gamma\right)\To \Gamma\To 1.$$
This extension is unique and in particular does not depend on the choice of $\mathbb{P}$. If the action is already free on $X$, then this extension is nothing but the one corresponding to the finite étale cover $X\to X/\Gamma$.
\end{lem}
\begin{proof}
Let $B_\Gamma$ be the classifying space of $\Gamma$ and $E_\Gamma\to B_\Gamma$ be its universal cover. Universal properties of the classifying space ensure the existence of canonical maps $X\times \mathbb{P}\to X\times E_\Gamma$ and $(X\times \mathbb{P})/\Gamma\to B_\Gamma$ making the corresponding diagram commutative. It shows readily that both group extensions are the same.

If the action is free on $X$, we can use the projection onto the first factor
$$\left(X\times \mathbb{P}\right)/\Gamma\To X/\Gamma.$$
It is a fibre bundle with fibre $\mathbb{P}$ and, using the homotopy exact sequence, we readily infer that
$$\pi_1\left((X\times \mathbb{P})/\Gamma\right)\simeq \pi_1(X/\Gamma).$$
\end{proof}

\begin{proof}[Proof of Corollary \ref{cor:kahler linéaire}]
We use the notation introduced in the preceding statement and we are in the situation where $\rho$ is injective. We can moreover assume that the finite étale cover $Y\to X$ is Galois, its Galois group to be denoted $\Gamma$. We make the following observations:
\begin{enumerate}
\item the Shafarevich and Iitaka fibrations being functorial, the group $\Gamma$ acts on $W$, $B$ and the fibration $f$ is equivariant with respect to both actions. Let us note however that the action is in general no longer free on both $W$ and $B$.
\item the fundamental group of $W$ is isomorphic to the one of $Y$.
\end{enumerate}
The last assertion is a consequence of the torsion freeness of $\pi_1(Y)$. We have indeed an exact sequence
$$1\To \pi_1(F)_Y\To \pi_1(Y)\To \pi_1^{orb}(W)\To 1$$
where $F$ is the general fibre of $sh_\rho$ and the orbifold structure on $W$ is induced by the fibration $sh_\rho$. The defining property of $F$ being the finiteness of $\pi_1(F)_Y$, we infer that this group is trivial. Finally the orbifold fundamental group $\pi_1^{orb}(W)$ is an extension
$$1\To K\To \pi_1^{orb}(W)\To \pi_1(W)\To 1$$
where $K$ is a group generated by torsion elements. As before, it implies that $K=1$ and that $\pi_1(Y)\simeq\pi_1(W)$.

We can now apply Theorem \ref{th:th principal}: $W$ can be deformed to a projective manifold $W_{alg}$ on which the group $\Gamma$ acts. To deal with the lack of freeness of the action of $\Gamma$ on $W_{alg}$, let us introduce a simply connected projective manifold $\mathbb{P}$ on which $\Gamma$ acts freely: such a manifold exists according to \cite{Ser} (see also \cite[Chapter IX, \S 4.2]{shafbook}). We can finally consider the quotient $X_{alg}:=\left(W_{alg}\times \mathbb{P}\right)/\Gamma$ as in Lemma \ref{lem:gp extension}: this is a smooth projective variety whose fundamental group is the extension
\begin{equation}\label{eq:extension group}
1\To \pi_1(W_{alg})\To \pi_1(X_{alg})\To \Gamma\To 1.
\end{equation}
Since the deformation is ($\Gamma$-equivariantly) topologically trivial, the extension (\ref{eq:extension group}) is the same when $W_{alg}$ is replaced with $W$. The Shafarevich map being $\Gamma$-equivariant and inducing an isomorphism between fundamental groups, we can plug $Y$ instead of $W_{alg}$ in the extension (\ref{eq:extension group}). Lemma \ref{lem:gp extension} shows that this extension is the same as the one corresponding to the cover $Y\to X$. It gives the desired isomorphism
$$\pi_1(X)\simeq \pi_1(X_{alg})$$
and ends the proof of Corollary \ref{cor:kahler linéaire}.
\end{proof}
\bibliographystyle{amsalpha}
\bibliography{bib_kod}

\providecommand{\bysame}{\leavevmode\hbox to3em{\hrulefill}\thinspace}
\providecommand{\MR}{\relax\ifhmode\unskip\space\fi MR }
\providecommand{\MRhref}[2]{%
  \href{http://www.ams.org/mathscinet-getitem?mr=#1}{#2}
}
\providecommand{\href}[2]{#2}
\begin{thebibliography}{CCE15}

\bibitem[BP13]{BP13}
Patrick Brosnan and Gregory Pearlstein, \emph{On the algebraicity of the zero
  locus of an admissible normal function}, Compos. Math. \textbf{149} (2013),
  no.~11, 1913--1962.

\bibitem[Bro82]{Br}
Kenneth~S. Brown, \emph{Cohomology of groups}, Graduate Texts in Mathematics,
  vol.~87, Springer-Verlag, New York-Berlin, 1982.

\bibitem[Buc06]{B06}
Nicholas Buchdahl, \emph{Algebraic deformations of compact {K}\"ahler
  surfaces}, Math. Z. \textbf{253} (2006), no.~3, 453--459.

\bibitem[Buc08]{B08}
\bysame, \emph{Algebraic deformations of compact {K}\"ahler surfaces. {II}},
  Math. Z. \textbf{258} (2008), no.~3, 493--498.

\bibitem[Cam81]{C81}
F.~Campana, \emph{Cor\'eduction alg\'ebrique d'un espace analytique faiblement
  k\"ahl\'erien compact}, Invent. Math. \textbf{63} (1981), no.~2, 187--223.

\bibitem[Car57]{Car}
Henri Cartan, \emph{Quotient d'un espace analytique par un groupe
  d'automorphismes}, Algebraic geometry and topology, Princeton University
  Press, Princeton, N. J., 1957, A symposium in honor of S. Lefschetz,,
  pp.~90--102.

\bibitem[CCE14]{CCE2}
Fr{\'e}deric Campana, Beno{\^{\i}}t Claudon, and Philippe Eyssidieux,
  \emph{Repr\'esentations lin\'eaires des groups k\"ahl\'eriens et de leurs
  analogues projectifs}, J. \'Ec. polytech. Math. \textbf{1} (2014), 331--342.

\bibitem[CCE15]{CCE1}
\bysame, \emph{Repr\'esentations lin\'eaires des groupes k\"ahl\'eriens:
  factorisations et conjecture de {S}hafarevich lin\'eaire}, Compos. Math.
  \textbf{151} (2015), no.~2, 351--376.

\bibitem[Del68]{D68}
P.~Deligne, \emph{Th\'eor\`eme de {L}efschetz et crit\`eres de
  d\'eg\'en\'erescence de suites spectrales}, Inst. Hautes \'Etudes Sci. Publ.
  Math. (1968), no.~35, 259--278.

\bibitem[Gra16]{Graf}
Patrick Graf, \emph{Algebriac approximation of k\"ahler treefolds of kodaira
  dimension zero}, preprint arXiv:1601.04307, 2016.

\bibitem[Kod63]{Kod}
K.~Kodaira, \emph{On compact analytic surfaces. {II}, {III}}, Ann. of Math. (2)
  77 (1963), 563--626; ibid. \textbf{78} (1963), 1--40.

\bibitem[M{\'e}g10]{thesedamien}
Damien M{\'e}gy, \emph{Sections hyperplanes \`a singularit\'es simples et
  exemples de variations de structure de hodge}, Ph.D. thesis, Institut Fourier
  (Grenoble), 2010, available at
  \texttt{https://tel.archives-ouvertes.fr/tel-00502378/}.

\bibitem[Nak99]{Nak99}
Noburu Nakayama, \emph{Compact k{\"a}hler manifolds whose universal convering
  spaces are biholomorphic to ${\Bbb c}^n$}, preprint RIMS-1230, 1999.

\bibitem[Sch12]{Sch}
Christian Schnell, \emph{Complex analytic {N}\'eron models for arbitrary
  families of intermediate {J}acobians}, Invent. Math. \textbf{188} (2012),
  no.~1, 1--81.

\bibitem[Ser58]{Ser}
Jean-Pierre Serre, \emph{Sur la topologie des vari\'et\'es alg\'ebriques en
  caract\'eristique {$p$}}, Symposium internacional de topolog\'\i a algebraica
  {I}nternational symposium on algebraic topology, Universidad Nacional
  Aut\'onoma de M\'exico and UNESCO, Mexico City, 1958, pp.~24--53.

\bibitem[Ser06]{S06}
Edoardo Sernesi, \emph{Deformations of algebraic schemes}, Grundlehren der
  Mathematischen Wissenschaften [Fundamental Principles of Mathematical
  Sciences], vol. 334, Springer-Verlag, Berlin, 2006.

\bibitem[Sha13]{shafbook}
Igor~R. Shafarevich, \emph{Basic algebraic geometry. 2}, third ed., Springer,
  Heidelberg, 2013, Schemes and complex manifolds, Translated from the 2007
  third Russian edition by Miles Reid.

\bibitem[Voi02]{V02}
Claire Voisin, \emph{Th\'eorie de {H}odge et g\'eom\'etrie alg\'ebrique
  complexe}, Cours Sp\'ecialis\'es [Specialized Courses], vol.~10, Soci\'et\'e
  Math\'ematique de France, Paris, 2002.

\bibitem[Voi04]{V04}
\bysame, \emph{On the homotopy types of compact {K}\"ahler and complex
  projective manifolds}, Invent. Math. \textbf{157} (2004), no.~2, 329--343.

\bibitem[Voi06]{V06}
\bysame, \emph{On the homotopy types of {K}\"ahler manifolds and the birational
  {K}odaira problem}, J. Differential Geom. \textbf{72} (2006), no.~1, 43--71.

\bibitem[Zuc79]{Zuc}
Steven Zucker, \emph{Hodge theory with degenerating coefficients. {$L_{2}$}
  cohomology in the {P}oincar\'e metric}, Ann. of Math. (2) \textbf{109}
  (1979), no.~3, 415--476.

\end{thebibliography}
\end{document}